\documentclass[11pt]{article}
\usepackage{amsmath,fullpage,amsthm}
\usepackage{amscd}
\usepackage{amsthm,mathtools} \usepackage{amssymb}
\usepackage{latexsym}
\usepackage{eufrak}
\usepackage{euscript}
\usepackage[T1]{fontenc}

\usepackage[width=18cm,height=25cm]{geometry}
\usepackage{epsfig}
\usepackage{tikz}
\usepackage{graphics}
\usepackage{todonotes}
\usepackage{array}
\usepackage{enumerate} \usepackage{authblk}

\newtheorem{theorem}{Theorem}[section]

\newtheorem{corollary}[theorem]{Corollary}

\newtheorem{problem}[theorem]{Problem}
\theoremstyle{definition}

\theoremstyle{proof}
\theoremstyle{observation}

\newtheorem{proposition}{Proposition}

\begin{document}
  \setcounter{Maxaffil}{2}
   \title{Co-maximal subgroup graph characterized by forbidden subgraphs}

  \author[a]{Pallabi Manna\thanks{mannapallabimath001@gmail.com}}

   \author[b ]{Santanu Mandal\thanks{santanu.vumath@gmail.com}}

   \author[c]{ Manideepa Saha\thanks{manideepasaha1991@gmail.com}}

\affil[a ]{Harish Chandra Research Institute,}
\affil[ ]{Prayagraj-211019, India}

   \affil[ b]{School of Computing Science and Engineering,}
  \affil[ ]{Vellore Institute of Technology,}
   \affil[ ]{Bhopal - 466114, India}

\affil[c]{Department of Mathematics,}
\affil[ ]{Presidency University,}
\affil[ ]{College Street, Kolkata-700073, India}

   \maketitle
\begin{abstract}
In this communication, the co-maximal subgroup graph $\Gamma(G)$ of a finite group $G$ is examined when $G$ is a finite nilpotent group, finite abelian group, dihedral group $D_n$, dicyclic group $Q_{2^n}$, and $p$-group. We derive the necessary and sufficient conditions for $\Gamma(G)$ to be a cluster graph, triangle-free graph, claw-free graph, cograph, chordal graph, threshold graph and split graph.  For the case of finite nilpotent group, we are able to classify it entirely. Moreover, we derive the complete structure of finite abelian group $G$ such that $\Gamma(G)$ is a split graph. We leave the readers with a few unsolved questions.
\end{abstract}


\section{Introduction}
In graph theory, forbidden subgraphs play a significant role in identifying various classes of graphs. Several graph classes including bipartite graphs, trees, perfect graphs, cographs, claw-free graphs, line graphs, chordal graphs, split graphs, cluster graphs etc. can be characterized in terms of different forbidden subgraphs. Let $\Gamma'=(V, E)$ be a graph and $S\subset V$. Then $G[S]$ is said to be an induced subgraph of $\Gamma'$ if the vertex set of $G[S]$ is $S$ and edge set contains the edges that have both end points in $S$. If a graph $\Gamma'$ does not contain a graph $F$ as its induced subgraph then $\Gamma'$ is called a $F$-free graph; or in other words $F$ is the forbidden subgraph of $\Gamma'$. In this paper we consider the graph classes namely cluster graphs, triangle-free graphs, claw-free graphs, cographs, chordal graphs, split and threshold graphs.


 We completely classify all finite nilpotent groups $G$ such that its co-maximal subgroup graph $\Gamma(G)$ belongs to the class of  cluster graphs, triangle-free graphs, claw-free graphs, cographs and threshold graphs. Moreover, we analyze the nilpotent groups which are not of prime power order  such that the co-maximal subgroup graph is a chordal graph as well as a split graph. We derive the complete structure of finite abelian groups $G$ such that $\Gamma(G)$ is a split graph. \\
Let us recall the definition of nilpotent group from group theory. A group $G$ is called a nilpotent group if it is the direct product of its Sylow $p$-subgroups for all primes $p$ dividing the order of the group. Let us consider the Dihedral group $D_{n}$ of order $2n$ and the Dicyclic group $Q_{2^n}$, where $n \geq 3$. $D_{n}$ has the following well-known representation: $\langle r, s: r^{n}=s^{2}=1, srs^{-1}=r^{-1}\rangle$. We investigate the values of $n$ such that $\Gamma(D_{n})$ belongs to either of the classes considered in this paper. Based on what we've studied about $\Gamma(D_{n})$, we are able to conclude that $\Gamma(D_{n})$ satisfies following.\\
(i) a triangle-free graph if and only if $n$ is a power of an odd prime.\\
(ii) a claw-free graph if and only if $n=4$.\\
(iii) a cograph if and only if $n$ is a power of an odd prime.\\
(iv) a chordal graph if and only if either $n=4$ or $n$ is a power of an odd prime.\medskip

Let $Q_{2^n}=\langle r, s: r^{2^{n-1}}=1, r^{2^{n-2}}=s^{2}, srs^{-1}=r^{-1}\rangle$. We find the values of $n$ such that $\Gamma(Q_{2^n})$ lies in one of the classes considered in this work (given later). We are able to make the conclustion: $\Gamma(Q_{2^n})$ is a triangle-free graph, cograph and a chordal graph if and only if $n$ is a power of some primes. There is no $n$ for which $\Gamma(Q_{2^n})$ is either a cluster graph or a claw-free graph. In this context, we first prove the following proposition. Here we use the notation '$(a,b)=c$' to denote the gcd of $a$ and $b$ equal to $c$.
\begin{proposition}
Let $Q_{2^n}=\{r, s: r^{2^{n-1}}=1, r^{2^{n-2}}=s^{2}, srs^{-1}=r^{-1}\}$. The non-trivial subgroups of $Q_{2^n}$ are either of Type-I: $\langle r^{d}\rangle$, where $d|2^{n-1}$; or of Type-II: $\langle r^{d}, r^{i}s \rangle$, where $d|2^{n-1}, 0\leq i \leq 2^{n-2}$. Then we have the following:\\
a) a vertex of Type-I say $\langle r^{d_1} \rangle$ is adjacent to a vertex of Type-II say $\langle r^{d_2}, r^{i}s \rangle$ if and only if $(d_{1}, d_{2})=1$;\\
b) any two vertices of Type-II say $\langle r^{d_1}, r^{i}s \rangle$ and $\langle r^{d_2}, r^{j}s \rangle$ are adjacent if and only if either of the following holds:\\
i) $(d_1, d_2)=1$;\\
ii) $(d_1, d_2)=2$ and $(i-j)$ is odd.
\end{proposition}

\begin{proof}
a) Clearly, any two vertices of Type-I are non-adjacent. So, let $H=\langle r^{d_1}\rangle$ and $K=\langle r^{d_2}, r^{i}s \rangle$ such that $H \sim K$. Now, $r \in HK$ gives there exist two positive integers say $m, n$ for which $md_{1}+nd_{2}=1$. This implies $(d_1, d_2)=1$. Similarly one can obtain $r\in HK$ when $(d_1, d_2)=1$. \\
b) Suppose $H=\langle r^{d_1}, r^{i}s \rangle$ and $K=\langle r^{d_2}, r^{j}s \rangle$ such that $H\sim K$. Let us assume that $d=(d_{1}, d_{2})$. Then we get $\langle r^{d}\rangle \subset HK$, where $d$ is the least power of $r$ such that $r^d$ can be expressed as product of $r^{d_1}$ and $r^{d_2}$. We claim that $d\in \{1, 2\}$. Otherwise, since $r, r^{2}$ are expressible as product of the powers of $r^{d_1}, r^{d_2}, r^{i}s, r^{j}s$. \\
Let $r=(r^{d_1})^{x_1}(r^{i}s)^{x_2}(r^{d_2})^{x_3}(r^{j}s)^{x_4}$. Since the elements of the form $r^{i}s, r^{j}s$ are of order $4$; so $x_{2}, x_{4}\in \{1, 2, 3\}$. Therefore we get the following equations:\\
1) $d_{1}x_{1}+d_{2}x_{3}\equiv 1$ (mod $2^{n-1}$), if $x_{2}=x_{4}=2$;\\
2) $d_{1}x_{1}+d_{2}x_{3}+(i-j)+2^{n-2}\equiv 1$ (mod $2^{n-1}$), if $x_{2}=x_{4}=1$ or $3$;\\
3) $d_{1}x_{1}+d_{2}x_{3}+(i-j)\equiv 1$ (mod $2^{n-1}$), if $x_{2}=1, x_{4}=3$ or $x_{2}=3, x_{4}=1$.\\
Similarly if $r^{2}=(r^{d_1})^{y_1}(r^{i}s)^{y_2}(r^{d_2})^{y_3}(r^{j}s)^{y_4}$ then we obtain the following equations:\\
4) $d_{1}y_{1}+d_{2}y_{3}\equiv 2$ (mod $2^{n-1}$), if $y_{2}=y_{4}=2$;\\
5) $d_{1}y_{1}+d_{2}y_{3}+(i-j)+2^{n-2}\equiv 2$ (mod $2^{n-1}$), if $y_{2}=y_{4}=1$ or $3$;\\
6) $d_{1}y_{1}+d_{2}y_{3}+(i-j)\equiv 2$ (mod $2^{n-1}$), if $y_{2}=1, y_{4}=3$ or $y_{2}=3, y_{4}=1$.\\
By (4)-(1) we have $d_{1}u+d_{2}v\equiv 1$ (mod $2^{n-1}$) for two positive integers $u, v$. This implies that $d$ divides $1$ as $d$ divides $d_{1}u+d_{2}v$. This leads to a contradiction that $d=(d_{1}, d_{2})$.\\
Let $d=2$ and $(i-j)$ is an even number. Then as the L. H. S of equation (2) is an even number so we get a contradiction to equation (2). \\
Thus either $(d_{1}, d_{2})=1$ or $(d_{1}, d_{2})=2$ and $(i-j)$ is odd.
\end{proof}

We use the notations $|G|$, $|s|$ and $\langle a \rangle$ to mean the order of a group $G$ and an element $s$ of a group and a cyclic group generated by the element $a$ resp. The symbol $G\times H$ indicates the direct product of two groups $G$, $H$. The notations $K_{n}, P_{n}, C_{n}, K_{1, 3}, 2K_{2}$ stand for a complete graph on $n$ vertices, a path on $n$-vertices, a cycle of length $n$, a claw and $2$ copies of $K_{2}$ (a complete graph with $2$ vertices) respectively. Apart from the $n$-length cycle, we also use the same notation $C_{n}$ to denote a cyclic group of order $n$. From the context it will be clear which we intend to mean actually. We now think back to a well-known theorem relating to line graphs.
\begin{theorem}
\label{line}
\cite{Godsil}
A graph $G$ is the line graph of some graph if and only if if $G$ does not have any one of the following $9$ Beineke graphs as its induced subgraph.
$$\begin{array}{ccccccccccccc}
   \includegraphics[scale=0.75]{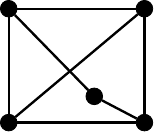}
   &&&\includegraphics[scale=0.75]{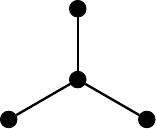}
   &&& \includegraphics[scale=0.75]{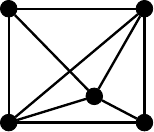} 
  &&&\includegraphics[scale=0.75]{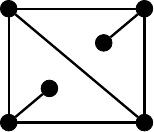}   
  &&& \includegraphics[scale=0.75]{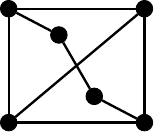}\\
     \Gamma_1&&&\Gamma_2&&&\Gamma_3&&& \Gamma_4&&&\Gamma_5\\\\
  \includegraphics[scale=0.75]{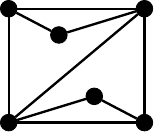}
  &&&\includegraphics[scale=0.75]{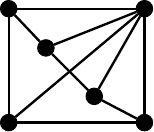}
  &&&\includegraphics[scale=0.75]{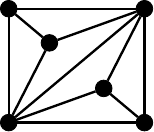}
  &&&\includegraphics[scale=0.75]{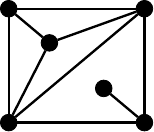}\\
   \Gamma_6&&&\Gamma_7&&&\Gamma_8&&&\Gamma_9\\
\end{array}$$
\end{theorem}
The rest of the paper is organized as follows. In section $2$, we provide complete characterizations of a finite nilpotent group $G$ for $\Gamma(G)$ to be cluster graph. We prove that $\Gamma(D_n)$ and $\Gamma(Q_{2^n})$ never produce a cluster graph. In section $3$, we identify all finite nilpotent groups, Dihedral groups $D_n$ and Dicyclic groups $Q_{2^n}$ for which their co-maximal subgroup graphs are triangle-free. Section $4$ describes all the nilpotent groups, Dihedral groups $D_n$ and Dicyclic groups $Q_{2^n}$ for which their co-maximal subgroup graphs are claw-free. In section $5$, we classify all finite nilpotent groups, $p$-groups, Dihedral groups $D_n$ and Dicyclic groups $Q_{2^n}$ for which their co-maximal subgroup graphs are cographs. In section $6$, we make the full characterizations of finite nilpotent groups, finite abelian groups, Dihedral groups $D_n$ and Dicyclic groups $Q_{2^n}$ for which their co-maximal subgroup graphs are chordal. Lastly in section $7$, we classify all finite nilpotent groups $G_1$ and all finite abelian groups $G_2$ such that $\Gamma(G_1)$ is a threshold graph and $\Gamma(G_2)$ is a split graph.


\section{Cluster Graph}
A graph is called a cluster graph if it does not contain $P_{3}$ ( a path on $3$-vertices) as its induced subgraph.\\
In general, a cluster graph is formed from the disjoint union of complete graphs. It is the complement of the complete multipartite graphs. In the cluster graph, adjacency is an equivaence relation, and their connected components are the equivalence classes for this relation. The cluster graph is a subclass of a cograph and a claw-free graph.
\begin{theorem}
Let $G$ be any finite nilpotent group. Then $\Gamma(G)$ is a cluster graph if and only if $G$ is any one of the following  groups:\\
a) a cyclic group of prime power order;\\
b) $C_{p}\times C_{p}$ ($p$ is a prime);\\
c) $Q_{8}$;\\
d)$G$ is a non-cyclic $p$-groups such that all maximal subgroups of $G$ is cyclic;
e) $C_{pq}$, where $p, q$ are two distinct primes.
\end{theorem}

\begin{proof}
Firstly, let $\Gamma(G)$ be a cluster graph. Let $|\pi(G)|\geq 2$ and $p, q, r$ be $3$ distinct prime divisors of $|G|$. Then $\Gamma(G)$ contains a $3$-vertex path $P\sim M\sim N$, where $P$ is the Sylow $p$-subgroup of $G$ and $M, N$ are two distinct maximal subgroups of $G$ such that $[G:M]=p, [G:N]=q$. Hence, we conclude that $|\pi(G)|\in \{1, 2\}$. Now two cases arise.\\
\textbf{Case 1.} $|\pi(G)|=1$\\
Let $|G|=p^{a}$ ($a \geq 1$). If $a=1$ then $G\cong C_{p}$. Clearly $\Gamma(C_{p})$ is a cluster graph. Therefore $a \geq 2$.\\
Now if $G$ is a cyclic group then the proof is obvious.\\
Suppose $G$ is non-cyclic group such that $a \geq 2$. If $G$ has unique minimal subgroup then $G$ is the dicyclic group $Q_{2^n}$ ($n \geq 3$). For $n>3$ there exist two non-trivial subgroups (say, $A, B$) of $Q_{2^n}$ of order $4$ and a maximal cyclic subgroup say $H$ of order $2^{n-1}$. Then $\Gamma(Q_{2^n})$ contains a path $A\sim H\sim B$. Therefore, $G$ must be $Q_{8}$.\\
On the other hand, if $G$ has order $p^{2}$ then any non-trivial subgroup is maximal. Hence, any $3$ of these subgroups form a triangle. So, $\Gamma(G)$ is a cluster graph in this case.\\
Let $G$ be a group such that $G$ has at least one maximal subgroups say $ N$ which is non-cyclic. Then $N$ is at least $2$-generated. Since, $G$ is not cyclic so there exists another maximal subgroup say $M$ of $G$. Now, $N$ is at least $2$-generated so we obtain an element say $x \in N\setminus M$ of order $p$ such that $\langle x \rangle \subset N$. This implies $\Gamma(G)$ contains a path $\langle x \rangle \sim M \sim N$. \\
So, $G$ must be a non-cyclic $p$-group in which every maximal subgroups are cyclic.\\
Therefore, if $|\pi(G)|=1$ then $\Gamma(G)$ is a cluster graph if $G$ is either a cyclic $p$-group or $C_{p}\times C_{p}$ or $Q_{8}$.\\
\textbf{Case 2.} $|\pi(G)|=2$\\
Let $p, q||G|$ be two distinct primes and $G=P\times Q$, where $P, Q$ are Sylow $p$- and $q$-subgroup of $G$ such that $|P|=p^{a}, |Q|=q^{b}$ ($a, b \geq 1$). If both $a, b \geq 2$ then $\Gamma(G)$ contains the path $P \sim M\sim N$, where $[G:M]=p, [G:N]=q$. Therefore, any one of $a, b$ must be equal to $1$. Let $|G|=p^{a}q$ ($a \geq 1$).\\
If $a=1$ then $G\cong C_{pq}$. \\
Let $a \geq 2$ and $M$ be a maximal subgroup of $G$ with $[G:M]=p$. Then there exists a subgroup say $H\subset G\setminus M$ whose order is $p$. So $\Gamma(G)$ contains the path $P\sim M \sim H$. Thus in this case we obtain $G\cong C_{pq}$.\\
Converse part:\\
Let $G$ be a $p$-group in which every maximal cyclic subgroups are cyclic. Suppose, $M, N$ are two maximal subgroups of $G$. Since, both $M, N$ are cyclic and $|M\cap N|=p^{a-2}$, where $|G|=p^{a}$. Hence, there are $(p^{a-1}-p^{a-2})$ are in $m$ as well as $N$. Clearly, these elements of $M$ are the generators of $M$ and for $N$ these elements generates $N$. Thus, there does not exist any path on $3$ vertices because if such path exists say $A\sim M\sim N$ then $A$ must be satisfies that $A\subset N$ and $A$ is not contained in $M$. But in this case $A=N$. Therefore $\Gamma(G)$ is a cluster graph.\\
For the other groups the proof is obvious.
\end{proof}
\begin{theorem}
$\Gamma(D_{n})$ is never be a cluster graph.
\end{theorem}
\begin{proof}
Let $r, s$ be two generators of $D_{n}$ such that order of $s$ is $2$ and $p, q$ be two primes dividing $n$. Let $n=pq$ ($p>q$). Then $\Gamma(D_{n})$ contains a path $\langle r^{p}, s\rangle \sim \langle r^{q}, s \rangle \sim \langle r^{p}, rs\rangle$. Similar manner we can prove that if $n$ has two or more distinct prime divisors then $\Gamma(D_{n})$ always has $P_{3}$. This implies that $n$ must be of prime power. Again, if $n=p^{k}$ ($k \geq 1$) then $\Gamma(D_{n})$ contains a path $\langle r^{p}, s \rangle \sim \langle r \rangle \sim \langle r^{p}, rs \rangle$. Therefore there does not exists $n$ for which $\Gamma(D_{n})$ is a cluster graph.
\end{proof}

\begin{theorem}
$\Gamma(Q_{2^n})$ ($n \geq 3$) is never be a cluster graph.
\end{theorem}

\begin{proof}
Let $Q_{2^n}$ be generated by $x, y$, where $y$ has order $4$. Suppose, $p>q$ are two distinct primes such that $n=pq$. Then $\Gamma(Q_{2^n})$ contains a $3$-vertex path $\langle x^{p}, y \rangle \sim \langle x^{q} \rangle \sim \langle x^{p}, xy \rangle$. On the other hand, if $n$ is a power of $p$ then $\Gamma(Q_{4n})$ carries the path $\langle x^{p}, y \rangle \sim \langle x \rangle \sim \langle x^{p}, xy \rangle$. This completes the proof of the theorem.
\end{proof}

   \section{Triangle-free}
A triangle-free graph is a graph that does not contain a cycle of length $3$ as its induced subgraph. It is a subclass of a bipartite graph. In this section we explore the finite nilpotent groups $G$ such that $\Gamma(G)$ is triangle-free. Moreover, we conclude that $\Gamma(D_{n})$ is a triangle-free graph if and only if $n=p^k$, where $p$ is an odd prime. But in case of $Q_{2^n}$ we have $\Gamma(Q_{2^n})$ is a triangle-free graph if and only if $n$ is a power of a prime.
\begin{theorem}
Let $G$ be a finite nilpotent group. Then $\Gamma(G)$ is triangle-free if and only if $G$ is either a cyclic $p$-group or $G \cong C_{p^{a}q^{b}}$, where $a, b \geq 1$.
\end{theorem}

\begin{proof}
Let $G\cong P_{1}\times P_{2}\times \cdots \times P_{r}$, where $r \geq 2$ and $P_{i}$ are Sylow $p_{i}$-subgroups of order $p_{i}^{\alpha_{i}}$ for $1\leq i \leq r$, such that $\Gamma(G)$ is triangle-free.\\
If $r \geq 3$, then $\Gamma(G)$ contains a triangle with vertices: $H_{1}, H_{2}, H_{3}$ where $|H_{1}|=p_{1}^{\alpha_{1}}p_{2}^{\alpha_{2}}, |H_{2}|=p_{2}^{\alpha_{2}}p_{4}^{\alpha_{4}}\cdots p_{r}^{\alpha_{r}}$, and $|H_{3}|=p_{1}^{\alpha_{1}}p_{3}^{\alpha_{3}}\cdots p_{r}^{\alpha_{r}}$. Thus, $G$ is either  i) a $p$-group or ii) $|G|=p_{1}^{\alpha_{1}}p_{2}^{\alpha_{2}}$, where both $\alpha_{1}, \alpha_{2}\geq 1$.\\
i) First we consider that $G$ is a $p$-group. Clearly, $G$ cannot have more than one maximal subgroup; elsewhere by Kulakoff theorem $G$ contains atleast $3$ maximal subgroups of $G$. They form together a triangle in $\Gamma(G)$. Hence, the maximal subgroup of $G$ is unique, which implies $G$ must be cyclic $p$-group.\\
ii) Next let $|G|$ has exactly two distinct prime divisors.\\
Suppose at least one of $\alpha_{i}>1$ (say, $\alpha_{1}>1$). Let $P_{i}$ ($i=1, 2$) be the Sylow $p_{i}$-subgroups of $G$ resp. We claim that both Sylow subgroups must be cyclic. If possible let $P_{1}$ be non-cyclic. Consider a subgroup $H\subset P_{1}$ of order $p_{1}^{\alpha_{1}-1}$. Then there exists an element say $x \in P_{1}$ of order $p_{1}$ with $\langle x \rangle \cap H=\{e\}$. Set $K=\langle x \rangle$. Then $\Gamma(G)$ contains a triangle $A, B, P_{1}$ such that $|A|=p_{1}^{\alpha_{1}-1}p_{2}^{\alpha_{2}}, |B|=p_{1}p_{2}^{\alpha_{2}}$
with $H \subset A, K\subset B$. Hence, both Sylow subgroups of $G$ must be cyclic and $G\cong C_{p^{a}q^{b}}$, where $a, b \geq 1$.\\
Converse is clear.
\end{proof}
We simply express the following conclusion as a result of this theorem.
\begin{corollary}
    If $G$ is a finite nilpotent group, then $\Gamma(G)$ is bipartite if and only if $G$ is either a cyclic $p$-group or $G \cong C_{p^{a}q^{b}}$, where $a, b \geq 1$.
\end{corollary}

\begin{theorem}
$\Gamma(D_{n})$ is a triangle-free graph if and only if $n$ is a power of an odd prime.
\end{theorem}

\begin{proof}
Let $\Gamma(D_{n})$ be a triangle-free graph.\\
First suppose, $p, q$ be two primes such that $n=pq$ with $p>q$. Then $\Gamma(D_{n})$ contains a triangle $\langle r^{p}, s \rangle \sim \langle r^{q}, s \rangle \sim \langle r \rangle \sim \langle r^{p}, s \rangle $, where $r, s$ are two generators of $D_{n}$ with order of $s$ is $2$. Similarly, if $n$ has two or more distinct prime divisors we can construct a triangle in $\Gamma(D_{n})$. On the other hand, let $n$ be a power of $2$. Now $\Gamma(D_{n})$ comprises a triangle by the $3$ vertices as $\langle r^{2}, rs \rangle, \langle r^{2}, s \rangle, \langle r \rangle$. Therefore $n=p^{k}$ (where, $p$ is an odd prime).\\
Conversely, let $n$ be a power of an odd prime.\\
If possible let $A, B, C$ be the vertices of a triangle in $\Gamma(D_{n})$. Without loss of generality, let $A=\langle r^{p}, s\rangle$. Then $B$ must be $\langle r \rangle$ as $p$ is odd. So, $C$ has only possibility that $\langle r^{p^m}, s \rangle$ (where $m >1$) or $\langle r^{p^l}, rs \rangle$, where $l \geq 1$. But in that case $A\nsim C$. Hence $\Gamma(D_{n})$ never contains a triangle.
\end{proof}

\begin{theorem}
$\Gamma(Q_{2^n})$, where $n \geq 3$,  does not contain a triangle if and only if $n$ is a prime power.
\end{theorem}

\begin{proof}
Suppose, $\Gamma(Q_{2^n})$ is triangle-free.\\
Let $Q_{2^n}$ be generated by $x, y$, with $|y|=4$, and $p, q$ be two primes such that $n=pq$ (with $p>q$).\\
Now, $\Gamma(Q_{2^n})$ contains a triangle with $3$ vertices as $\langle x^{p}, y \rangle, \langle x \rangle$ and $\langle x^{q}, y \rangle$. Thus $n$ must be a prime power.\\
To prove the converse, let $n=p^{k}$. Contrarily let, $\Gamma(Q_{2^n})$ carries a triangle $A, B, C$. Without loss of generality, let $A=\langle x^{p}, y\rangle$. Then $B, C$ must have the possibilities $\langle x \rangle, \langle x^{p}, xy \rangle$ respectively. But $A\nsim C$, which contradicts that $A, B, C$ are the vertices of a triangle. Hence, in this case $\Gamma(Q_{2^n})$ is triangle-free.
\end{proof}


\section{Claw-free}  
A graph is said to be claw-free if it does not contain the complete bipartite graph $K_{1, 3}$ as its induced subgraph. The classes of graphs which are claw-free includes a complete graph ($K_{n}$), a cycle graph ($C_{n}$), a path graph ($P_{n}$), rook graph, line graph, strongly perfect graph, antiprism graphs, cocktail party graph etc. We are now interested to analyze the finite nilpotent groups $G$ for which $\Gamma(G)$ is claw-free.
\begin{theorem}
    Let $G$ be a finite nilpotent group. Then $\Gamma(G)$ is claw-free if and only if $G$ is one of the following:\\
    a) If $|\pi(G)|=1$, then $G$ is either a cyclic $p$-group or $C_{p}\times C_{p}$ or $Q_{8}$, where $p$ is a prime;\\
    b) If $|\pi(G)|=2$, then $G$ is the group either $C_{pq}$ or $C_{p^{a}q}$ ($a \geq 2$), where $p, q$ are the distinct primes dividing $o(G)$.
\end{theorem}

\begin{proof}
Let $G$ be a finite nilpotent group such that $\Gamma(G)$ is claw-free.\\
First we observe that $|\pi(G)|\leq 2$.\\
If $|G|=p_{1}^{\alpha_{1}}p_{2}^{\alpha_{2}}p_{3}^{\alpha_{3}}$, then $\Gamma(G)$ contains a claw with $3$ pendent vertices as: $A, B, C$ such that $|A|=p_{1}p_{2}^{\alpha_{2}}p_{3}^{\alpha_{3}}, |B|=p_{1}p_{3}^{\alpha_{3}}, |C|=p_{2}^{\alpha_{2}}p_{3}^{\alpha_{3}}$ and the central vertex say $H$ is of order $p_{1}^{\alpha_{1}}p_{2}^{\alpha_{2}}$. Similarly if $|G|$ has $4$ and above distinct prime divisors then $\Gamma(G)$ contains a claw.\\
a) Suppose $|\pi(G)|=1$.\\
Then $G$ must be a $p$-group of order $p^{a}$. If $G$ is cyclic then the result holds trivially. \\
Now consider $G$ as non-cyclic $p$-group. Let $M$ be a maximal subgroup of $G$. Then $|M|=p^{a-1}$ and $|G\setminus M|=p^{a-1}(p-1)$. If $G$ has more than one subgroup of order $p$ then $G\setminus M$ contains $p^{a-1}$ subgroup of order $p$. Otherwise, $G$ will be a dicyclic group $Q_{2^n}$ ($n \geq 3$). Let $G\cong Q_{2^n}$. Here, if $n \geq 4$ then there exist a unique cyclic maximal subgroup say $A$ and atleast $3$ cyclic subgroup of order $4$ such that they form a claw along with the central vertex as $A$. But if $n=3$ then $G \cong Q_{8}$. Since every cyclic subgroup of $Q_{8}$ are order $4$ and they are maximal so any three form a triangle. Hence in this case $\Gamma(G)$ is claw-free.\\
Let $G$ be the group whose minimal subgroup is not unique. Now if $a \geq 3$ then $|M|=p^{a-1}$ and $G \setminus M$ contains at least $3$ distinct subgroups of order $p$. Thus $\Gamma(G)$ contains a claw together with the central vertex $M$. But, when $a=2$ then $G \cong C_{p}\times C_{p}$. In this case all non-trivial subgroups are of order $p$. So, any non-trivial subgroup is a maximal subgroup of $G$. Therefore any $3$ of them form a triangle which implies that $\Gamma(G)$ is claw-free.\\
b) Next, let $|\pi(G)|=2$.\\
In this case, $|G|=p^{a}q^{b}$, where $a, b \geq 1$. Here we observe that both $a, b$ can't be  $\geq 2$. If possible let, $a \geq 2, b \geq 2$. Then $\Gamma(G)$ contains a claw with $3$ pendent vertices are of orders $p^{a-1}q^{b}, q, q^{b}$ and central vertex is of order $p^{a}q^{b-1}$. Thus, at least one of $a, b$ must be $1$. Let $a \geq 2$ and $b=1$ and $|G|=p^{a}q$.\\
We claim that $G$ must be a cyclic group. \\
For the sake of contradiction let  $G$ be non-cyclic and $P, Q$ be the Sylow $p$-subgroup, Sylow $q$-subgroup of $G$ resp. Clearly the Sylow $q$-subgroup $Q$ of $G$ must be cyclic. We show that the Sylow $p$-subgroup is also cyclic. If possible let, Sylow $p$-subgroup $P$ of $G$ is non-cyclic. Then $G$ has at least two elements of order $p$ say $a, b$ such that $\langle a \rangle \cap \langle b \rangle =\{e\}$. Let $Q=\langle c \rangle$. So, $\Gamma(G)$ contains a claw with pendent vertices $A, B, C$ as: $\langle ac \rangle, \langle bc \rangle, \langle c \rangle$ resp. and the central vertex $H$ is $P$. Therefore both Sylow subgroups of $G$ is cyclic and hence $G\cong C_{p^{a}q}$.\\
Again, if $a=b=1$ then $G$ is the cyclic group $C_{pq}$.\\
\textbf{Converse Part:}\\
If $G$ is either a cyclic $p$-group or $C_{pq}$ or $C_{p^{a}q}$ it can be easily prove that $\Gamma(G)$ is claw-free.\\
If $G\cong C_{p}\times C_{p}$, then every non-trivial subgroup is of order $p$ and any $3$ such subgroups form a triangle. Thus, $\Gamma(G)$ fails to form a claw.\\
Similarly if $G \cong Q_{8}$, one can easily check that $\Gamma(G)$ is claw-free.
\end{proof}
To draw the conclusion $\Gamma(G)$ to be a line graph or not, first of all $\Gamma(G)$ should be a claw-free graph. Then by the above theorem and the Theorem \ref{line} we get the following conclustion.
\begin{corollary}
    If $G$ is a finite nilpotent group, then $\Gamma(G)$ is a line graph of some graph if and only if $G$ is either a cyclic $p$-group or $C_{p}\times C_{p}$ or $Q_{8}$ or $C_{pq}$ or $C_{p^{a}q}$ ($a\geq 2$), where $p, q$ are two distinct primes.
\end{corollary}
\begin{theorem}
 $\Gamma(D_{n})$ is a claw-free graph if and only if $n=4$.
\end{theorem}
\begin{proof}
Firstly, let $\Gamma(D_{n})$ be a claw-free graph.\\
Let $p, q$ be two primes with $n=pq$ (where $p>q$), and $D_{n}$ be generated by $r, s$ where $|s|=2$. Now, $\Gamma(D_{n})$ contains a claw with the pendant vertices as $\langle r^{p}, s \rangle, \langle r^{p}, r^{2}s \rangle$ and $\langle r^{p} \rangle$ and the central vertex is $\langle r^{q}, s \rangle$. Similarly, if $n$ has more than two prime divisors then we can prove that $\Gamma(D_{n})$ contains a claw. Thus, $n=p^k$.\\
Let $p$ be odd. Again we obtain that $\Gamma(D_{n})$ contains a claw in which $3$ pendant vertices are $\langle r^{p}, s \rangle, \langle r^{p}, rs \rangle, \langle rs \rangle$ respectively and the central vertex is $\langle r \rangle$. Thus in this case $\Gamma(D_{n})$ is never be a claw-free graph.\\
Now consider $n=2^{k}$. If $k \geq 3$, then $\Gamma(D_{n})$ contains a claw with the central vertex as $\langle r \rangle$ and $3$ pendant vertices as: $\langle r^{2}, rs \rangle, \langle r^{4}, r^{3}s \rangle, \langle rs \rangle$. This gives $k=2$ and $n=4$.\\
Conversely, let $n=4$. The non-trivial subgroups are either $\langle r \rangle$, or $\langle r^{2}\rangle$, or $\langle r^{2}, s\rangle$, or $\langle r^{2}, rs\rangle$, or $\langle s \rangle$, or $\langle rs \rangle$.  But for any choices of the vertices we obtain that $\Gamma(D_{4})$ does not contain a claw. 
\end{proof}

\begin{theorem}
$\Gamma(Q_{2^n})$, where $n \geq 3$, is not a claw-free graph.
\end{theorem}

\begin{proof}
Let $p$ be a prime such that $n=p^k$. Suppose, $x, y$ are two generators of $Q_{2^n}$ with $|y|=4$. Then $\Gamma(Q_{2^n})$ contains a claw where the $3$ pendant vertices are $\langle x^{p}, y \rangle, \langle x^{p}, xy \rangle$ and $\langle xy \rangle$ resp. and the central vertex is $\langle x \rangle$. In a similar manner one can easily find out a claw in $\Gamma(Q_{2^n})$ if $n$ has more than one prime divisors.
\end{proof}


\section{Cograph}
A graph is called a cograph if it does not contain a path on $4$ vertices, that is $P_{4}$, as its induced subgraph.\\
A cograph can be constructed from a single vertex by repeating applications of disjoint union and complement. The diameter of a cograph is $2$. It is a subclass of a perfect graph and it contains the class of Threshold graph, cluster graph, complete graph, complete bipartite graph etc. In this section we characterize all finite nilpotent groups $G$ such that $\Gamma(G)$ is a cograph. Let us start with the groups $D_{n}$, $Q_{n}$ and  find the values of $n$ such that $\Gamma(D_{n})$ and $\Gamma(Q_{2^n})$ is a cograph.
\begin{theorem}
\label{Comax_co_dih}
$\Gamma(D_{n})$ is a cograph if and only if $n$ is a power of an odd prime.
\end{theorem}

\begin{proof}
First suppose, $\Gamma(D_{n})$ be a cograph. Let $r, s$ be two generators of $D_{n}$ with $|s|=2$, and $p, q$ be two primes such that $n=pq$. Now $\Gamma(D_{n})$ contains a $4$-vertex path $\langle r^{p}\rangle \sim \langle r^{q}, rs \rangle \sim \langle r^{p}, rs \rangle \sim \langle r^{q}\rangle$. Similarly if $n$ has more than two distinct prime divisors or $n$ is product of powers of $p$ and $q$ then we can show that $\Gamma(D_{n})$ contains a $4$-vertex path. This implies that $n$ must be a power of some prime.\\
Let $n=2^{m}$, where $m\geq 2$. Here $\Gamma(D_{n})$ contains a path $\langle s \rangle \sim \langle r^{2}, rs \rangle \sim \langle r^{2}, s \rangle \sim \langle rs \rangle$. Thus, $n$ must be a power of an odd prime.\\
Conversely, let $n=p^{m}$, where $p$ is an odd prime and $m \geq 1$. Contrarily, let $A\sim B\sim C\sim D$ be a path in $\Gamma(D_{n})$. Without loss of generality, let $A=\langle r^{p}, s\rangle$. Then $B, C$ must be $\langle r \rangle$ and $\langle r^{p}, rs \rangle$ resp. So, $D$ must be either of the form $\langle r^{p^k}\rangle$ or $\langle r^{p^k}, r^{i}s \rangle$. But any of these choices for $D$, we have $C\nsim D$. Thus $\Gamma(D_{n})$ is a cograph in this case.
\end{proof}

\begin{theorem}
\label{Comax_co_dicyclic}
 $\Gamma(Q_{2^n})$ is a cograph if and only if $n$ is a prime power.
\end{theorem}

\begin{proof}
Let, $\Gamma(Q_{2^n})$ be a cograph. Suppose $p, q$ are two primes with $n=pq$ (where $p>q$) and $x, y$ are two generators of $Q_{2^n}$. If $n=pq$ then $\Gamma(Q_{2^n})$ contains a path $\langle x^{p}\rangle \sim \langle x^{q}, y \rangle \sim \langle x^{p}, xy \rangle \sim \langle x^{q} \rangle$. Thus $n=p^{k}$ (where $p$ is a prime).\\
Conversely, let $n=p^{k}$. From the adjacency condition of any two non-trivial subgroups it is clear that $\Gamma(Q_{2^{p^k}})$ is a cograph.
\end{proof}

\begin{proposition}
\label{prop_cograph}
Let $G$ be a non-cyclic finite group. If $\Gamma(G)$ is a cograph then $G$ must be $2$-generated.
\end{proposition}

\begin{proof}
Let $\Gamma(G)$ be a cograph. If possible let, $G=\langle x, y, z\rangle$ and $H_{1}=\langle x \rangle, H_{2}=\langle y \rangle, H_{3}=\langle z \rangle$. Then $\Gamma(G)$ contains a $4$-vertex induced path $\langle z \rangle \sim \langle x, y \rangle \sim \langle y, z \rangle \sim \langle x \rangle$. This leads to a contradiction that $\Gamma(G)$ is a cograph. Hence, $G$ must be $2$-generated.
\end{proof}
\medskip 
The converse of the above proposition is not true in general (see the generalized quaternion group). So, the question is for which cases the converse hold?
A group is called an EPO-group if every non-identity elements of the group are of prime order. We now show that the converse of the above proposition is also holds for an EPO-group.
\begin{theorem}
Let $G$ be an EPO-group. Then $\Gamma(G)$ is a cograph if and only if $G$ is $2$-generated.
\end{theorem}

\begin{proof}
Suppose $G$ is an EPO-group such that $\Gamma(G)$ is a cograph. Clearly by Proposition \ref{prop_cograph}, $G$ must be $2$-generated.\\
Conversely, let $G$ be a group generated by two elements say $x, y$ with $o(x)=p, o(y)=q$. Let us consider two cases.\\
\textbf{Case 1.} $p=q$.\\
Then $G$ is the group $C_{p}\times C_{p}$. Let us assume that $\Gamma(G)$ contains a path $A\sim B\sim C\sim D$. Now every non-trivial subgroups of $G$ are of order $p$ and hence any two such subgroups generate $G$. Thus, in the $4$-vertex path we obtain $A\sim C$. This implies $\Gamma(G)$ does not contain any $4$-vertex induced path. So, in this case $\Gamma(G)$ is a cograph.\\
\textbf{Case 2.} $p\neq q$.\\
Let $p>q$. Then $G$ is the group $C_{p}\rtimes C_{q}=\langle x, y \rangle$, where $o(x)=p, o(y)=q$. Now the Sylow $p$-subgroup $\langle x \rangle$ of $G$ is normal. So, if $\Gamma(G)$ contains a path say $A\sim B\sim C\sim D$ then the path is not an induced path. This implies that $\Gamma(G)$ is a cograph. Similarly, if $p<q$ then $\Gamma(G)$ is a cograph. 
Therefore, if $G$ is a $2$-generated EPO-group then $\Gamma(G)$ is a cograph.
\end{proof}

\begin{theorem}
\label{Comax_co_p}
Let $G$ be a finite abelian group of prime power order. Then $\Gamma(G)$ is a cograph if and only if $G$ is either a cyclic group of prime power order or $C_{p^k}\times C_{p^r}$ (where $k, r \geq 1$).
\end{theorem}

\begin{proof}
Let $p$ be a prime that divides $|G|$ and $G\cong C_{p^{r_1}}\times C_{p^{r_2}}\times \cdots \times C_{p^{r_k}}$, where $r_{i}\geq 1$ for all $i=1, 2, \cdots, k$. Let $\Gamma(G)$ be a cograph. Then by Proposition \ref{prop_cograph}, $G$ is either cyclic or a $2$-generated abelian $p$-group. Hence, $G$ is either a cyclic group of prime power order or $G \cong C_{p^{r_1}}\times C_{p^{r_2}}$, where $r_{1}, r_{2}\geq 1$.\\
Converse Part: \\
If $G$ is a cyclic group of prime power order then clearly $\Gamma(G)$ is a cograph.\\
Let $G \cong C_{p^k}\times C_{p^r}$ (where $k, r \geq 1$). Let $x, y$ be two generators of $G$. Suppose, $\Gamma(G)$ contains a $4$-vertex induced path $H_{1}\sim H_{2}\sim H_{3}\sim H_{4}$. First we observe that any two subgroups of $G$ of the form $C_{p^{k_1}}\times C_{p^{r_1}}$ and $C_{p^{k_2}}\times C_{p^{r_2}}$, where $k_1, k_2<k, r_1, r_2<r$, are not adjacent. Because, let the first one is $H$ and the second one is $K$ then $H\cap K$ is non-trivial and $HK\subset G$.\\
Without loss of generality, let $H_{1}=C_{p^k}\times C_{p^l} (l\geq 0)$. Then, $H_{2}, H_{3}$ and $H_{4}$ will be $C_{p^m}\times C_{p^r}$, $C_{p^k}\times C_{p^n}$ and $C_{p^s}\times C_{p^r}$ respectively, where $m, n, s \geq 0$. But then $H_{1}\sim H_{4}$. Thus, such induced path does not exist. Hence, $\Gamma(G)$ is a cograph.
\end{proof}

\begin{theorem}
\label{Comax_nil_co}
Let $G$ be a finite nilpotent group such that $|\pi(G)| \geq 2$. Then $\Gamma(G)$ is a cograph if and only if $G$ is a cyclic group $C_{p^{a}q^{b}}$ (where $a, b \geq 1$), where $p, q \in \pi(G)$ are distinct primes. 
\end{theorem}
\begin{proof}
Let $G$ be a finite nilpotent group and $|\pi(G)| \geq 3$. Let $p, q \in \pi(G)$ be two distinct primes. Now $G$ has at least two maximal subgroups say $M, N$ with $[G:M]=p$ and $[G:N]=q$. Suppose, $P, Q$ are Sylow $p$- and $q$-subgroups of $G$. Then $\Gamma(G)$ contains a $4$-vertex path $P\sim M\sim N\sim Q$. Hence, $|\pi(G)|=2$. \\
Let $G=P\times Q$ such that $|P|=p^{a}, |Q|=q^{b}$. Each non-trivial subgroup of $G$ is of order either a power of $p$ or a power of $q$ or of the form $p^{i}q^{j}$, where $1\leq i \leq a, 1\leq j \leq b$ and $(i, j)\neq (a, b)$. \\
Firstly, we claim that the Sylow $p$-subgroup $P$ of $G$ must be cyclic. Contrarily, let $P$ be non-cyclic. Then there exist a subgroup say $P'$ of order $p^{a-1}$ and an element $a$ of order $p$ such that $a \notin P$. Then $\Gamma(G)$ contains a path $P' \sim \langle a \rangle \times Q \sim P\times \langle b \rangle \sim Q$, where $o(b)=q$. Thus $P$ must be cyclic. \\
Using the similar argument we conclude that $Q$, the Sylow $q$-subgroup of $G$, is also cyclic. Therefore $G$ must be a cyclic group $C_{p^{a}q^{b}}$ (where $a, b \geq 1$).\\
\textbf{Converse Part:}\\
Let $G\cong C_{p^{a}q^{b}}$ ($a, b \geq 1$), where $p, q$ are two distinct prime divisors of $|G|$. Any non-trivial subgroups of $G$ are of orders either power of $p$ or power of $q$ or of the form $p^{i}q^{j}$, where $1\leq i \leq a, 1\leq j \leq b$ and $(i, j)\neq (a, b)$. If possible let $\Gamma(G)$ contain a path $A, B, C, D$.\\
a) Let $B \cong C_{p^k}$ ($k\leq a$). If $k<a$ then  $A, C$ must be $C_{p^{a-k}q^{b}}$, which is not possible. If $k=a$ then one of $A, C$ must be $C_{q^b}$ and the other is $C_{p^{a-1}q^{b}}$. Without loss of generality, let $A\cong C_{q^b}$ and $C\cong  C_{p^{a-1}q^{b}}$. Then $D$ has the only possibility $C_{p^{a}q^{b-1}}$. Thus $A\sim D$. \\
b) If $B$ is a power of $q$ then just as done in a) we get that there does not exist any $4$-vertex induced path in $\Gamma(G)$.\\
c) Let $B \cong C_{p^{a}q^{b-1}}$. Without loss of generality, let $A \cong C_{q^b}$ and $C\cong C_{p^{a-1}q^{b}}$. Then $D$ must be $C_{p^a}$. But in this case $A\sim D$.\\
Thus in all cases we reach at the conclusion that $\Gamma(G)$ does not contain a $4$-vertex induced path. Therefore, $\Gamma(G)$ is a cograph.\\
Other cases are obvious.
\end{proof}
 
From the above Proposition \ref{prop_cograph} and Theorem \ref{Comax_nil_co} we obtain the following corollary.
\begin{corollary}
Let $G$ be a finite abelian group. Then $\Gamma(G)$ is a cograph if and only if $G$ is either a cyclic group of prime power order or an abelian group $C_{p^k}\times C_{p^r}$ (where, $k, r \geq 1$) or a cyclic group $C_{p^{a}q^{b}}$ (where, $a, b \geq 1$). Provided, $p, q$ are distinct primes.
\end{corollary}

\begin{proposition}
Let $G$ be a group of order $pq$, where $p\neq q$ are two primes. Then $\Gamma(G)$ is a cograph.
\end{proposition}

\begin{proof}
Since $p, q$ are two primes without loss of generality we assume that $p>q$. Let $P$ be Sylow $p$-subgroup of $G$. Since, $p>q$, the Sylow $p$-subgroup $P$ of $G$ is normal. Now we consider here two cases based on the normality of Sylow $q$-subgroup of $G$. \\
Case 1. Let the Sylow $q$-subgroup of $G$ be normal.\\
In this case $G$ is the nilpotent group. Hence, $G\cong P\times Q$ (where, $Q$ is the Sylow $q$-subgroup of $G$) i.e., $G\cong C_{pq}$. Therefore by Theorem \ref{Comax_nil_co} $\Gamma(G)$ is a cograph.\\
Case 2. Let Sylow $q$-subgroup of $G$ be not normal.\\
Here $G$ has $p$ distinct subgroups of order $q$ say $Q_{1}, Q_{2}, \cdots, Q_{p}$. In this case we show that $\Gamma(G)$ is a cograph. For the sake of contradiction, let $\Gamma(G)$ contains a $4$-vertex induced path say $A\sim B\sim C\sim D$. Since the non-trivial subgroups of $G$ must be of order either $p$ or $q$ so first consider $o(A)=p$ i.e., $A\cong P$ . Then $B\cong Q_{1}$. This implies $o(C)$ must be $p$. But as $G$ has unique subgroup of order $p$ so $C=A$. Hence this path is not possible.\\
On the other hand, if $o(A)=q$ then $B\cong P$ and $o(C)=q$ also. Let $A\cong Q_{1}$ and $C\cong Q_{2}$. But then $D\cong P$ (as o(D)=p and $G$ contains unique subgroup of order $p$). Thus in any cases such a $4$-vertex induced path is not possible. Hence, $\Gamma(G)$ is a cograph.
\end{proof}

\begin{theorem}
Let $G$ be a group of order $p^{2}q$. Then $\Gamma(G)$ is a cograph if and only if $G$ is either $C_{3}\rtimes C_{4}$ or $A_{4}$ or $C_{q}\rtimes C_{p^2}$ or $C_{p^2}\rtimes C_{q}$ or a group whose every non-trivial subgroups are EPPO group.
\end{theorem}

\begin{proof}
Let $G$ be a group of order $p^{2}q$ and $P$ be the Sylow $p$-subgroup of $G$.
Since $p, q$ are distinct primes we consider two cases.\\
Case 1. $p>q$.\\
Since $p>q$ so it is clear that the Sylow $p$-subgroup $P$ of $G$ must be normal.
Let $G$ contains a subgroup of order $pq$ say $K$. Now we claim that $P$ must be cyclic. Elsewhere, $P$ contains atleast two subgroups of order $p$ say $P_{1}, P_{2}$. Then $\Gamma(G)$ contains the path $P_{1}\sim K\sim P\sim Q$, where $P_{2}\subset K$ and $Q$ is the Sylow $q$-subgroup of $G$. Hence, $P\cong C_{p^2}$. Thus $G$ is either $C_{p^2}\times C_{q}$ or $C_{p^2}\rtimes C_{q}$.\\
Next let $G$ does not contain any subgroup of order $pq$. Then each subgroup of $G$ is of order either $p$ or $p^{2}$ or $q$. But a subgroup of order $q$ is adjacent to the subgroup $P$ and a subgroup of order $p$ is neither adjacent to $P$ nor adjacent to a subgroup of order $q$. Thus in any $4$-vertex path only $P$ and the subgroups of order $q$ will appear. But as $P$ is unique so such $4$-vertex induced path is not possible.\\
Therefore, in this case we get $G$ is either $C_{p^2}\times C_{q}$ or $C_{p^2}\rtimes C_{q}$ or a group whose every non-trivial subgroups are EPPO group.
Case 2. $p<q$.\\
 Let $n_{p}, n_{q}$ be the number of Sylow $p$-subgroups and Sylow $q$-subgroups of $G$. Now $n_{p}\in \{1, q\}$ and $n_{q}\in \{1, p, p^{2}\}$. Since $p<q$ so $n_{q}\neq p$. Thus $n_{q}\in \{1, p^{2}\}$. Also one can observe that $(n_{p}, n_{q})\neq (q, p^{2})$. \\
i) If both $n_{p}=n_{q}=1$ then $G$ is nilpotent. So $G\cong P\times Q$, where $Q$ is the Sylow $q$-subgroup of $G$. We now claim that $P$ must be cyclic. Otherwise $P$ contains at least two subgroups say $P_{1}, P_{2}$ of order $p$. As $G$ has a subgroup of order $pq$ then $\Gamma(G)$ contains a $4$-vertex induced path $P_{1}\sim K\sim P\sim Q$, where $P_{2}\subset K$. Hence $G\cong C_{p^2}\times C_{q}$. \\
ii) Let $n_{q}=1, n_{p}=q$. In this case just following the similar argument as done in above case we can conclude that if $G$ contains a subgroup of order $pq$ then $G$ must be $C_{q}\rtimes C_{p^2}$ or $G$ is a group whose every non-trivial subgroups must be an EPPO group.\\
iii) Let $n_{p}=1, n_{q}=p^{2}$. $n_{q}=p^2$ gives $q|(p^{2}-1)$. This implies $n_{q}|(p+1)$ (as $p<q$). Hence $q=p+1$ and so $p=2, q=3$. Therefore $G$ is the group of order $12$. Thus $G$ is either $A_{4}$ or $C_{3}\rtimes C_{4}$ or $D_{6}$. But by theorem \ref{Comax_co_dih} $\Gamma(D_{6})$ is not a cograph. So $G$ is either $A_{4}$ or $C_{3}\rtimes C_{4}$.\\
Conversely,\\
a) $G$ is the group $C_{p^2}\rtimes C_{q}$. Then $G$ has unique subgroup of order $p$ and $p^2$. To show that $\Gamma(G)$ is a cograph, contrarily let $A\sim B\sim C\sim D$ be a $4$-vertex induced path in $\Gamma(G)$. If $o(A)=pq$ then $o(B)=p^{2}, o(C)=q, o(D)=p^{2}$. Since the subgroup of order $p^2$ is unique so this path is impossible.\\
Again, if $o(A)=p^{2}$ then $o(B)=q$ or $pq$ and $o(C)=p^{2}$. So such path does not exist.\\
If $o(A)=q$ then $o(B)=p^{2}, o(C)=q$ or $pq$ and $o(D)=p^2$. Thus this path can not occur.\\
Therefore in any case $\Gamma(G)$ is a cograph.\\
b)Similarly for the groups $C_{q}\rtimes C_{p^2}$ and $C_{3}\rtimes C_{4}$ we can show that $\Gamma(C_{q}\rtimes C_{p^2})$ and $\Gamma(C_{3}\rtimes C_{4})$ are cographs.\\
c) If $G$ is a group whose every non-trivial subgroups are EPPO group then it is obvious that $\Gamma(G)$ is a cograph.\\
d) If $G$ is $A_{4}$ then its every non-trivial subgroups are EPPO and hence $\Gamma(A_{4})$ is a cograph.
\end{proof}


\section{Chordal Graph}
A graph is said to be a chordal graph if it does not contain any cycle of length $4$ and above. \\
Every chordal graph is a perfect graph. The classes of graph that are chordal includes a complete graph, complete bipartite graph, block graph, interval graph etc. We are now investigate the chordality of a finite nilpotent groups, $D_{n}$ and $Q_{2^n}$ as well. 
\begin{theorem}
\label{Comax_chordal_ab}
Let $G$ be a finite abelian group of prime power order. Then $\Gamma(G)$ is a chordal graph if and only if $G$ is one of the following groups:\\
a) a cyclic group of prime power order;\\
b) $C_{p^k}\times C_{p}$, where $p$ is a prime and $k \geq 1$ is a natural number;\\
c) $C_{p}\times C_{p}\times C_{p}$ (where $p$ is a prime).
\end{theorem}

\begin{proof}
Let $p\in \pi(G)$ be a prime and $G$ be a finite abelian group of order power of $p$ with $\Gamma(G)$ is a chordal graph. Let $G \cong C_{p^{r_1}}\times C_{p^{r_2}}\times \cdots \times C_{p^{r_k}}$. Clearly when $G$ is a cyclic group then the proof is obvious. So, we consider $G$ as a non-cyclic abelian group. Now $3$ cases arise depending on $r_{i}$ (where, $1 \leq i \leq k$).\\
\textbf{Case 1.} Let $r_{i}=1$ for all $i$.\\
In this case $G\cong (C_{p})^{k}$. Let $G$ be generated by the cyclic subgroup $H_{j}$ ($1 \leq j\leq k$), where $H_{j}=\langle (0,0,...,0,1,0,...,0)\rangle$ and $1$ is in $j$-th component. We claim that $k\leq 3$.\\
If possible let $k \geq 4$. Then $\Gamma(G)$ contains a $4$-cycle $A\sim B\sim C\sim D \sim A$ such that $A=\langle H_{1}, H_{3}, H_{4} \rangle$, $B=\langle H_{2}, H_{4},...,H_{k}\rangle$, $C=\langle H_{1}, H_{2}, H_{3}, H_{4}\rangle$, $D=\langle H_{2}, H_{5},..., H_{k}\rangle$. Hence, $G \cong C_{p}\times C_{p}\times C_{p}$ or $C_{p}\times C_{p}$.\\
\textbf{Case 2.} There exists one $r_{i}\geq 2$ and all others are $1$.\\
$G \cong C_{p^{r_1}}\times C_{p}\times \cdots \times C_{p}$. Just like Case-1, let $G$ be generated by $H_{i}$ (where, $1 \leq i \leq k$). We assert that $k =2$.\\
Suppose $k \geq 3$. Then $\Gamma(G)$ contains a $4$-vertex cycle $A\sim B\sim C\sim D \sim A$, where $A=\langle H_{1}, H_{2}\rangle, B=\langle H_{2}, H_{3}, \cdots , H_{k}\rangle, C=\langle H_{1}\rangle, D=\langle \overline{H_{1}}, H_{2}, \cdots , H_{k}\rangle$ and $\overline{H_{1}}=C_{p}$. Therefore in this case $G \cong C_{p^k}\times C_{p}$ ($k \geq 1$).\\
\textbf{Case 3.}  More than one $r_{i}\geq 2$ ($ 1\leq i \leq k$)\\
We claim that there does not exist any non-cyclic abelian group.\\
Let $k \geq 2$, $r_{1}, r_{2}\geq 2$ and $G$ be generated by $H_{1}, H_{2}, \cdots, H_{k}$, where $H_{i}$ ($1\leq i \leq k$) are defined just as in Case-1. Then $\Gamma(G)$ contains a $4$-vertex cycle $A\sim B\sim C\sim D \sim A$, where $A=\langle H_{1}, \overline{H_{2}} \rangle, B=\langle \overline{H_{1}}, H_{2}, \cdots, H_{k} \rangle, C=\langle H_{1}\rangle, D=\langle H_{2}, \cdots, H_{k}\rangle$, where $\overline{H_{1}}=\langle (p, 0, 0,\cdots, 0)\rangle, \overline{H_{2}}=\langle (0, p, 0,\cdots, 0) \rangle$. \\
Thus in all cases we get $G$ is one of the groups in (a), (b) and (c).\\
\textbf{Converse part:}\\ 
Let $G \cong C_{p^k}\times C_{p}$ ($k>1$) and $G=\langle H_{1}, H_{2}\rangle$, where $H_{1}=\langle (1, 0)\rangle, H_{2}=\langle (0, 1)\rangle$. If possible let $A_{1}\sim A_{2}\sim \cdots \sim A_{m}\sim A_{1}$ be a cycle in $\Gamma(G)$. Now $A_{1}=H_{1}$. So $A_{2}=H_{2}, A_{m}=\overline{H_{1}}H_{2}$, where $\overline{H_{1}}=\langle (p, 0)\rangle$. Then $A_{3}$ must be $H_{1}$ or $\langle H_{1}, H_{2}\rangle$. But both are not possible. This implies that there does not exist a cycle of length $\geq 4$. \\
Similarly if $k=1$ i.e $G\cong C_{p}\times C_{p}$ then as every non-trivial subgroups are maximal so any cycle of length $4$ and above always contains a chord. Hence $\Gamma(C_{p^k}\times C_{p})$ is a chordal graph.\\
Procceding in a similar approach we can prove that $\Gamma(C_{p}\times C_{p}\times C_{p})$ is a chordal graph.
\end{proof}
The above theorem tells about the chordality of the co-maximal subgroup graph of a finite abelian $p$-group. But for the non-abelian $p$-groups we can partially characterize some groups. So there arise an open question :
\begin{problem}
Characterize all non-abelian $p$-groups $G$ such that $\Gamma(G)$ is a chordal graph.
\end{problem}

\begin{theorem}
\label{Comax_nil_chordal}
Let $G$ be a finite nilpotent group such that $|\pi(G)|\geq 2$. Then $\Gamma(G)$ is chordal if and only if $G$ is either $C_{pqr}$ or $C_{p^{a}q}$ ($a\geq 1$) or $o(G)$ has exactly two distinct prime divisors such that one Sylow subgroup is cyclic group of prime order and other Sylow subgroup is a $2$-generated group of prime power order having prime exponent (where $p, q, r$ are three distinct primes).
\end{theorem}

\begin{proof}
Suppose $G$ is a finite nilpotent group which is not a $p$-group such that $\Gamma(G)$ is chordal. Let $G\cong P_{1}\times P_{2}\times \cdots \times P_{r}$, where $P_{i}$'s are the Sylow $p_{i}$-subgroups of order $p_{i}^{\alpha_{i}}$ for $1\leq i \leq r$.\\
First we observe that, $r \leq 3$. If possible let, $r\geq 4$. Then $\Gamma(G)$ contains a $4$-vertex cycle: $A, B, C, D, A$, where $|A|=p_{1}^{\alpha_{1}}p_{3}^{\alpha_{3}}, |B|=p_{1}^{{\alpha_{1}}-1}p_{2}^{\alpha_{2}}\cdots p_{r}^{\alpha_{r}}, |C|=p_{1}^{\alpha_{1}}p_{2}^{{\alpha_{2}}-1}p_{3}^{\alpha_{3}}\cdots p_{r}^{\alpha_{r}}, |D|=p_{2}^{\alpha_{2}}p_{4}^{\alpha_{4}}\cdots p_{r}^{\alpha_{r}}$. Hence $o(G)$ can have at most $3$ distinct prime divisors.\\
Let \textbf{$|\pi(G)|=3$} and $|G|=p^{a}q^{b}r^{c}$, where $a, b, c \geq 1$.\\
We claim that if any one of $a, b, c $ is larger than $1$ then $\Gamma(G)$ contains a $4$-vertex cycle.\\
Let $a>1$ and $|G|=p^{a}qr$. Now $\Gamma(G)$ contains a $4$-vertex cycle $A, B, C, D, A$, where $|A|=p^{a}, |B|=p^{a-1}qr, |C|=p^{a}r, |D|=qr$. Thus, $a=b=c=1$ and then $G\cong C_{pqr}$.\\
Let \textbf{$|\pi(G)|=2$} and $|G|=p^{a}q^{b}$.\\
Let $P, Q$ be the Sylow $p$- and  $q$-subgroups of $G$ resp. If both $a, b \geq 2$ then $\Gamma(G)$ contains a $4$-vertex cycle: $(A, B, C, D, A)$, where $|A|=p^{a}, |B|=p^{a-1}q^{b}, |C|=p^{a}q^{b-1}, |D|=q^{b}$. Hence at least one of $a, b$ is $1$ (say $b=1$). Then $o(G)=p^{a}q$.\\
If both Sylow subgroups are cyclic then $G \cong C_{p^{a}q}$.\\
Now, let the Sylow $p$-subgroup is non-cyclic. Then we observe that $P$ must be $2$-generated; otherwise if $G=\langle x, y, z, w\rangle$, where $o(x)=p^{l}, o(y)=p^{m}, o(z)=p^{n}, o(w)=q$, then $\Gamma(G)$ contains a $4$-vertex cycle: $(A, B, C, D, A)$ such that $A=\langle x, y\rangle, B=\langle z, w \rangle, C=\langle x, y, z \rangle$ and $D=\langle y, z, w \rangle$. Thus, Sylow $p$-subgroups of $G$ must be $2$-generated. Let $G=\langle x, y, z\rangle$ such that $o(x)=p^{l}, o(y)=p^{m}$ and $o(z)=q$. We claim that both $l, m$ must be equal to  $1$. Otherwise if any one of $l, m$ is greater than $1$ (say $l>1$) then $\Gamma(G)$ contains a $4$-vertex induced cycle: $(X, Y, Z, W, X)$, where $X=\langle x, y \rangle, Y=\langle x^{p}, y, z \rangle, Z=\langle x, z \rangle, W=\langle y, z \rangle$. Therefore $l=m=1$ and so $G$ is the group such that one Sylow subgroup is cyclic group of prime order and other is a $2$-generated $p$-group of exponent $p$.\\
\textbf{Converse Part:}\\
a) If $G$ is $C_{pqr}$ then the non-trivial subgroups are of order $p, q, r, pq, pr, qr$. One can easily check that $\Gamma(G)$ cannot have any cycle of length $4$, $5$ or $6$.\\
b) Let $G\cong C_{p^{a}q}$. If possible let $\Gamma(G)$ contain a cycle of length $4$ and above. Without loss generality, let the first vertex has the order $p^{a}$. The next vertex must be of order $p^{a-1}q$. This gives the third vertex has order $p$; elsewhere first vertex will be adjacent to third vertex. Now the $4$-th vertex has order divisible by $pq$. Then in any aspect we obtain first and $4$-th vertex are adjacent. Hence $\Gamma(G)$ cannot contains a cycle of length $4$ and above.
\end{proof}
Motivated from this theorem we simply state the corollary below.
\begin{corollary}
Let $G$ be a finite abelian group. Then $\Gamma(G)$ is chordal if and only if $G$ is either $C_{pqr}$ or $C_{p^{a}q}$ ($a\geq1$) or $C_{p}\times C_{pq}$ or $C_{p^k}$ or $C_{p}\times C_{p}\times C_{p}$ or $C_{p^k}\times C_{p}$ ($k\geq 1$) or a cyclic group of prime power order, where $p, q, r$ are distinct primes.
\end{corollary}

\begin{theorem}
$\Gamma(D_{n})$ is a chordal graph if and only if $n$ is either a power of an odd prime or $n=4$.
\end{theorem}

\begin{proof}
Let $\Gamma(D_{n})$ be a chordal graph.\\
Firstly, let $p, q$ be two primes with $n=pq$ (where $p>q$) and $r, s$ generate $D_{n}$, where $|s|=2$. For $n=pq$, $\Gamma(D_{n})$ contains a $4$-cycle $\langle r^{q}\rangle \sim \langle r^{p}, s\rangle \sim \langle r^{q}, s \rangle \sim \langle r^{p}, rs \rangle \sim \langle r^{q}\rangle$. Similarly if $n$ has $3$ or more prime divisors then one can obtain a $4$-cycle in $\Gamma(D_{n})$. This implies $n$ must be a prime power.\\
If $n=2^{k}$ such that $n \geq 8$, then $\Gamma(D_{n})$ also carries a $4$-cycle $\langle r^{2}, s\rangle \sim \langle r^{4}, rs \rangle \sim \langle r^{4}, r^{2}s \rangle \sim \langle r^{2}, rs \rangle \sim \langle r^{2}, s\rangle $. Therefore in this case $n=4$. Thus $n$ is either $4$ or a power of an odd prime.\\
\textbf{Converse Part:}\\
Suppose $n=p^{k}$, where $p$ is an odd prime. Since, in this case $\Gamma(D_{n})$ is a cograph (by Theorem \ref{Comax_co_dih}) so we only need to check whether a $4$-cycle exists or not in $\Gamma(D_{n})$. If possible let $\Gamma(D_{n})$ has a $4$-cycle say $A\sim B\sim C\sim D \sim A$. Without loss of generality, let $A=\langle r^{p}, s \rangle$. Then $B$ must be $\langle r \rangle$ as $p$ odd. So $C$ must be $\langle r^{p}, rs \rangle$. But for any choices of $D$ we have $C\nsim D$. This implies that $\Gamma(D_{p^k})$ is a chordal graph.\\
Similarly, we can show that $\Gamma(D_{4})$ is also a chordal graph. 
\end{proof}
\begin{theorem}
\label{Comax_chordal_dicyclic}
$\Gamma(Q_{2^n})$ is chordal if and only if $n$ is a prime power.
\end{theorem}
\begin{proof}
Let $\Gamma(Q_{2^n})$ is a chordal graph.\\
Firstly let $n=pq$, where $p, q$ are distinct primes with $p>q$, and $x, y$ be two generators of $Q_{2^n}$.  In this case $\Gamma(Q_{2^n})$ contains a $4$-cycle $A\sim B\sim C\sim D\sim A$ with $A, B, C, D$ as follows: $\langle r^{q}\rangle \sim \langle r^{p}, s\rangle \sim \langle r^{q}, s \rangle \sim \langle r^{p}, rs \rangle$ (respectively). Following the similar manner we can construct a $4$-cycle in case of $n$ is divisible by more than $2$ distinct prime divisors. This gives $n$ must be a prime power.\\
Conversely, let $n=p^{k}$. As $\Gamma(Q_{2^n})$ (see Theorem \ref{Comax_co_dicyclic}) is a cograph so we only check the existence of a $4$-cycle in this case. If possible let $A\sim B\sim C\sim D\sim A$ be a $4$-cycle. Clearly, form the adjacency condition we must have $A=\langle x^{p}, y\rangle, B=\langle x \rangle, C=\langle x^{p}, xy \rangle$. Then for any choices of $D$ we have $C\nsim D$. This proves that $\Gamma(Q_{2^n})$ is a chordal graph when $n=p^k$.
\end{proof}


\section{Threshold Graph and Split Graph}
A graph is called a Threshold Graph if it does not contain $\{P_{4}, C_{4}, 2K_{2}\}$ as its induced subgraph.\\
Threshold graph is closed under the operations of disjoint union and complement. This is one of the important subclass of a cograph, split graph. A Split graph is a graph whose vertex set can be partitioned into a clique and an independent vertex set. It forbids $\{C_{4}, C_{5}, 2K_{2}\}$. A split graphs are graphs that are both chordal and the complements of chordal graphs.  In this section we classify the finite nilpotent groups $G$ such that $\Gamma(G)$ is a threshold graph. Moreover, we explore all finite abelian groups for which its co-maximal subgroup graph is a split graph.

\begin{theorem}
Let $G$ be a finite nilpotent group. Then $\Gamma(G)$ is a threshold graph if and only if  $G$ is one of the following groups:\\
a) a cyclic group of prime power order;\\
b) $C_{p}\times C_{p}$;\\
c) $Q_{2^{p^k}}$ (where, $p^k \geq 3$ and $p$ is a prime);\\
d) $C_{p^{a}q}$, where $a \geq 1$.\\
Provided that $p, q$ are distinct prime divisors of $|G|$.
\end{theorem}

\begin{proof}
Let $\Gamma(G)$ be a threshold graph. Then it is $P_{4}$-free and $C_{4}$-free as well. According to the Theorems \ref{Comax_nil_co}, \ref{Comax_nil_chordal} and  \ref{Comax_chordal_dicyclic}, $G$ takes one of the form (a)-(d).\\
Conversely,\\
(a)If $G$ is a cyclic group of prime power order then the proof is obvious.\\
(b) Suppose $G\cong C_{p}\times C_{p}$. Now it is enough to show that $\Gamma(G)$ is $2K_{2}$-free. Since any non-trivial subgroup is maximal, therefore any $3$ of them form a triangle. Thus $\Gamma(G)$ is $2K_{2}$-free and hence it is a threshold graph.\\
(c) Let $G \cong Q_{2^{p^k}}$ (where $p^k \geq 3$ and $p$ is a prime) such that $G$ be generated by $x, y$ with the order of $y$ is $4$. Let us examine whether $\Gamma(G)$ is $2K_{2}$-free or not. If possible let $\{A, B\}$ and $\{C, D\}$ form $2K_{2}$. Every non-trivial subgroup of $G$ which is not maximal must be of order either $2$ or $4$. So, let $A=\langle x^{2}, y\rangle, C=\langle x^{2}, xy\rangle$. Then if any one of $B, D$ is $\langle x\rangle$ then $A, C\sim \langle x \rangle$. On the other hand, none of $B, D$ can be a subgroup of order $4$ from the adjacency condition. Similarly, for the other choices of vertices we also obtain that $\Gamma(G)$ is $2K_{2}$-free. Hence $\Gamma(G)$ is a threshold graph.\\
(d) Let $G$ be the cyclic group $C_{p^{a}q}$ ($a \geq 1$).  We now show that $\Gamma(G)$ is $2K_{2}$-free. Let $\{A_{1}, A_{2}\}, \{B_{1}, B_{2}\}$ form $2K_{2}$.\\
If $|A_{1}|=p^{a}$ then $|A_{2}|=p^{a-1}q$. Now $|B_{1}|$ is either $q$ or $p^{i}q$ ($i<a$) or $p^{j}$ ($j<a$). But $|B_{1}| \neq q$ as in this case $A_{1}\sim B_{1}$. For the other cases, $|B_{2}|$ must be either $p^{a}$ or $p^{a}q$. The latter case is impossible since the vertices of $\Gamma(G)$ are non-triivial. Again if $|B_{2}|=p^{a}$ then $B_{2}\cong A_{1}$. Thus in all aspect $\Gamma(G)$ is $2K_{2}$-free. Similarly, for the other possible choices of $\{A_{1}, A_{2}\}, \{B_{1}, B_{2}\}$ we conclude that $\Gamma(G)$ is $2K_{2}$-free.\\
 Therefore $\Gamma(G)$ is a threshold Graph.
\end{proof}

\begin{theorem}
Let $G$ be a finite abelian group. Then $\Gamma(G)$ is a split graph if and only if $G$ is either of the following groups:\\
(a) $C_{p^{a}q}$ ($a\geq 1$);\\
(b) $C_{p}\times C_{p}\times C_{q}$;\\ 
(c) $C_{p^k}\times C_{p}$ ($k\geq 1$ is a natural number);\\
(d) $C_{p}\times C_{p}\times C_{p}$;\\
(e) a cyclic group of prime power order.\\
Provided $p, q$ are three distinct primes.
\end{theorem}

\begin{proof}
Firstly, let $\Gamma(G)$ is a split graph. So $\Gamma(G)$ must be $\{C_{4}\}$-free. By Theorem \ref{Comax_chordal_ab} and \ref{Comax_nil_chordal}, $G$ must be any one of the groups: $C_{pqr}$ ($p<q<r$ are primes) or the groups in (a)-(e).\\
As the co-maximal subgroup graph of any of these groups $G$ is a chordal graph so $\Gamma(G)$ is also $C_{5}$-free.\\
But if $G\cong C_{pqr}$ then $\Gamma(G)$ contains $2K_{2}$ by the vertices $\{A, B\}$ and $\{C, D\}$, where $A\cong C_{pq}, B\cong C_{rs}, C\cong C_{pr}, D\cong C_{qs}$. This contradicts that $\Gamma(G)$ is a split graph. Hence $G$ is any one of the groups in (a)-(e).\\
Converse Part:\\
If $G$ is one of the groups (a)-(e) then one can easily check that $\Gamma(G)$ is $2K_{2}$-free. On the other hand, $\Gamma(G)$ is also $\{C_{4}, C_{5}\}$-free. Thus $\Gamma(G)$ is a split graph.
\end{proof}

\begin{problem}
Classify all the finite nilpotent groups whose co-maximal subgroup graph is a split graph.
\end{problem}

\section{Acknowledgements}
Pallabi Manna expresses gratitude to the Department of Atomic Energy (DAE) India for providing funding during this work. The  research of Santanu Mandal is supported by the University Grants Commission (UGC) India under the beneficiary code BININ01569755. Manideepa Saha acknowledges the funding of DST-SERB-SRG (Sanction no.
SRG/2019/000475 and SRG/2019/000684, Govt. of India)
\section{Statements and Declarations} \textbf{Competing Interests:} The authors made no mention of any potential conflicts of interest.

\end{document}